\documentclass[a4paper,11pt,onecolumn]{amsart}
\usepackage{amsfonts}

\usepackage{amscd}
\usepackage{amsmath}
\usepackage{amssymb}
\usepackage{amsthm}
\usepackage{epsf}
\usepackage{latexsym}
\usepackage{verbatim}
\usepackage{color}
\usepackage{todonotes}
\usepackage{enumitem}
\input epsf.tex
\usepackage{bbm}
\usepackage{mathtools}

\usepackage{epstopdf}
\usepackage{graphicx}
\usepackage{sidecap}

\theoremstyle{plain}
\newtheorem{theorem}{Theorem}
\newtheorem{definition}{Definition}[section]
\newtheorem{lemma}[definition]{Lemma}
\newtheorem{corollary}[definition]{Corollary}
\newtheorem{proposition}[definition]{Proposition}

\theoremstyle{definition}

\newtheorem*{thm2}{Theorem~2}

\theoremstyle{remark}

\newtheorem{examplex}{Example}
\newenvironment{example}
  {\pushQED{\qed}\examplex}
  {\popQED\endexamplex}

\title[Existence of Maximum Likelihood Estimates in ERGMs]{Existence of Maximum Likelihood Estimates in Exponential Random Graph Models}

\author{Henry Bayly}
\address{Boston College; Department of Mathematics; Maloney Hall; Chestnut Hill, MA 02467}
\email{baylyh@bc.edu}

\author{Aditya Khanna}
\thanks{A.K. was supported by NIH grants P20 GM 130414 and P30 AI 042853.} 
\address{Brown University; School of Public Health; 121 South Main Street
Providence, RI 02912}
\email{adityakhanna@brown.edu}

\author{Kathryn Lindsey}
\thanks{K.L. was partially supported by NSF grant \#1901247}
\address{Boston College; Department of Mathematics; 567 Maloney Hall; Chestnut Hill, MA 02467}
\email{lindseka@bc.edu}

\begin{document}

\begin{abstract}
We present a streamlined proof of the foundational result in the theory of exponential random graph models (ERGMs) that the maximum likelihood estimate exists if and only if the target statistic lies in the relative interior of the convex hull of the set of realizable statistics.  
\end{abstract}

\maketitle


\section{Introduction}

Exponential Random Graph Models (ERGMs) are a family of probability distributions $\textrm{Pr}_{\theta}$, parameterized by $\theta \in \mathbb{R}^p$, on the set $\mathcal{G}_k$ of all graphs with $k$ vertices.  ERGMs are used to model real-world networks in policy studies \cite{Robins2012}, global migration \cite{Windzio2018, Hancean2020}, education \cite{Chen2020, Wu2021},
scientific collaboration and knowledge creation \cite{10.1007/978-3-642-03367-4_25},
viral transmission through human networks \cite{Andalibi2021, Goodreau2012_PUMA, Jenness2016, Khanna2019, Zelner2020}. Notably, in the study of prevention of sexually transmitted diseases, common modeling techniques based on differential equations do not adequately capture the network structure of interest \cite{Carnegie2012, Khanna2014} and ERGMs enable more accurate simulation of epidemics \cite{Jenness2018}. 
Computational scientists use computed or estimated graph statistics to select a probability distribution $\textrm{Pr}_{\hat{\theta}}$ that assigns a higher probability to graphs whose structures more closely match the empirical target statistics; they then investigate phenomena on a sample of graphs selected randomly from $\mathcal{G}_k$ with respect to $\textrm{Pr}_{\theta}$. Open source tools to estimate and simulate networks are now widely used \cite{statnet, Hunter2008ergm, krivitsky2012package}.
Consequently, the question of existence and uniqueness of a parameter $\hat{\theta}$ so that $\textrm{Pr}_{\hat{\theta}}$ maximizes the likelihood of the target statistics is fundamental to the theory of ERGMs.

A well-known and foundational result in the theory of ERGMs is that 
 there exists $\hat{\theta}$ that maximizes the likelihood estimate if and only if the target statistics lie in the relative interior of the convex hull of the set of realizable statistics; furthermore if such a $\hat{\theta}$ exists, it is unique.  Multiple formulations and proofs of this  result exist in the literature (e.g. \cite{barndorff2014information, rinaldo2009geometry, brown1986fundamentals, Geyer, CM}).  \emph{The main contribution of this article is a streamlined proof of this result}; the goal is to enable scholars who are not experts in convex analysis to understand -- via a concise, readable, and mathematically rigorous proof -- why this important theorem is true.

 Our aim is to present the most direct and intuitive path possible to prove this result, avoiding any non-essential jargon and background.  Except for an appeal to a simple case of H\"{o}lder's Inequality (in the proof of Proposition \ref{p:concavity}), our proof uses only undergraduate-level real analysis.  Generalized versions of the result -- for example, applying to broader classes of exponential distributions or directed graphs -- requiring more delicate proof techniques are treated in the literature referenced above; in particular, the framework of \cite{barndorff2014information} applies to non-differentiable functions.


  

\begin{figure}
\centering
\includegraphics[width=.5\linewidth]{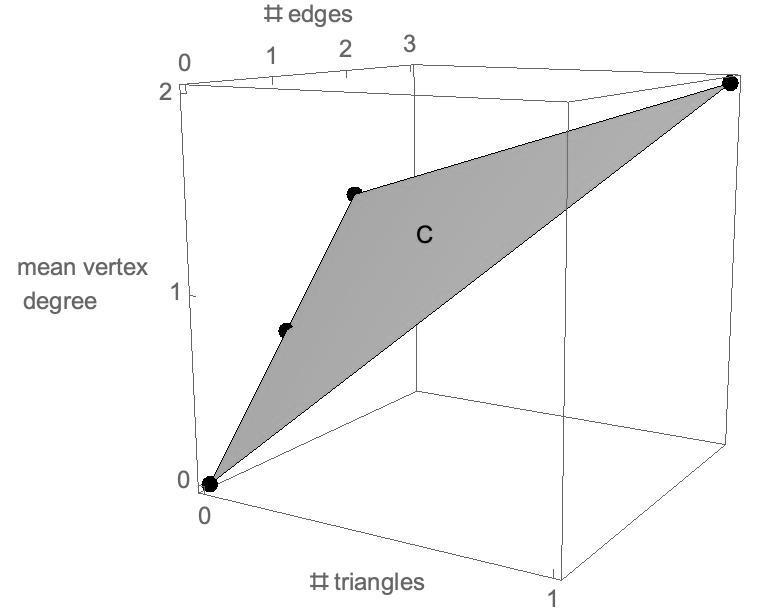}
\caption{Consider $z:\mathcal{G}_3 \to \mathbb{R}^3$ to be the statistics $(\# \textrm{triangles}, \# \textrm{edges}, \# \textrm{mean vertex degree})$.  The space $\mathcal{G}_3$ consists of $8$ graphs (which form 4 isomorphism classes).  
The set of realizable statistics,  $z(\mathcal{G}_3) \coloneqq \{z(g) \mid g \in \mathcal{G}_3\}$, are the 4 black points; the convex hull $C$ of this set is shown in gray. Because number of edges and mean vertex degree are linearly dependent, the set $C$ is not top-dimensional -- it is a 2-dimensional polytope embedded in $\mathbb{R}^3$. The interior of $C$ (viewed as a subset of $\mathbb{R}^3$) is the empty set; the relative interior of $C$ consists of all points in the interior of the gray triangle (viewed as a subset of its affine hull).}
\label{f:rint}
\end{figure}

\bigskip

\subsection{Statement and discussion of theorems}
Before stating the result precisely, we establish notation.  First, we specify that $\mathcal{G}_k$  is the set of all labeled, undirected graphs with $k$ vertices (``labeled'' means that we do not identify isomorphic graphs). Fix a function $z:\mathcal{G}_k \to \mathbb{R}^n$.  We think of each coordinate of  $z(g)$ as recording one statistic about the graph $g$ (for example, number of edges, number of triangles, measures of connectivity, etc.).  Define $C$ to be the convex hull (in $\mathbb{R}^n$) of $z(\mathcal{G}_k)$, the set of realizable statistics. Recall that, if we denote the vertices (extreme points) of $C$ by $v_1,\ldots,v_m$, the relative interior of $C$ may be defined as
\begin{equation} \label{eq:rintdef}
\textrm{rint}(C) \coloneqq \left \{\sum_{i=1}^m \lambda_i v_i \mid \lambda_i > 0, \sum_{i=0} \lambda_i^m = 1 \right \}.
\end{equation}
If the statistics that are the coordinate functions of $z$ are all linearly independent on $\mathcal{G}_k$, then $C$ is a polytope of dimension $n$ embedded in $\mathbb{R}^n$; if not, we let $V$ be the smallest linear subspace of $\mathbb{R}^n$ that contains a translate of $C$, i.e. $$V \coloneqq  \textrm{span}(\{z(g_1) - z(g_2) \mid g_1,g_2 \in \mathcal{G}_k\}),$$ and
set $V^{\perp} \coloneqq \{u \in \mathbb{R}^n \mid u \cdot v = 0 \textrm{ for all } v \in V\}$.
We define the family of probability distributions $\{\textrm{Pr}_{\theta} \mid \theta \in \mathbb{R}^n\}$ on $\mathcal{G}_k$ by 
$$\textrm{Pr}_{\theta}(g) \coloneqq \frac{e^{\theta \cdot z(g)}}{\sum_{\bar{g} \in \mathcal{G}_k} e^{\theta \cdot z(\bar{g})}}.$$ Fix $t \in \mathbb{R}^n$, the ``target statistic.''  The \emph{likelihood estimate} $L(\theta)$ is the weight assigned by $\textrm{Pr}_{\theta}$ to a graph $g$ such that $z(g) = t$, i.e. $L:\mathbb{R}^n \times \mathbb{R}^n \to (0,\infty)$ is the function 
$$L(\theta) \coloneqq  \frac{e^{\theta \cdot t }}{\sum_{\bar{g} \in \mathcal{G}_k} e^{\theta \cdot z(\bar{g})}}.$$ We will find it convenient to work with the logarithm of the likelihood; define $\ell:\mathbb{R}^n \to \mathbb{R}$ by 
$$\ell(\theta) \coloneqq \ln(L(\theta)).$$ 

We are now ready to state the well-known and foundational result about the existence and uniqueness of the maximum likelihood estimate for ERGMs.  We reformulate this result as Theorems \ref{t:mainthm} and \ref{t:degeneracy} below.

\begin{theorem} \label{t:mainthm}
If $dim(C) =n$, then
\begin{enumerate}
 \item the log likelihood function $\ell$ is strictly concave, and 
 \item there exists a parameter $\hat{\theta} \in \mathbb{R}^n$ such that $L(\hat{\theta})= \sup_{\theta \in \mathbb{R}^n} L(\theta)$ if and only if $t$ is in the interior of $C$; furthermore such a $\hat{\theta}$ is unique.  
  \end{enumerate}
  If $dim(C) < n$, then 
  \begin{enumerate}
 \item[(3)] the log likelihood function $\ell$ is concave but not strictly concave, 
 \item[(4)] there exists a parameter $\hat{\theta} \in V$ such that $L(\hat{\theta})= \sup_{\theta \in \mathbb{R}^n} L(\theta)$ if and only if $t$ is in the relative interior of $C$; furthermore there is a unique such $\hat{\theta}$ in $V$, and
 \item[(5)] $L(\theta) = L(\theta+u)$ for all  $\theta \in \mathbb{R}^n$ and $u \in V^{\perp}$. 
   \end{enumerate}

\end{theorem} 

\noindent The proof of Theorem \ref{t:mainthm} is the focus of \S \ref{s:concavityAndThm1}. 

Theorem \ref{t:degeneracy}, which is the focus of \S \ref{s:implications},  describes what happens when the target statistic $t$ is outside $C$ and implications for degeneracy of the associated ERGM model during the model training process.   There are various algorithms used in practice (most based on Markov chain Monte Carlo methods -- see e.g. \cite{Hunter2008ergm},\cite{jin2013fitting}) for seeking a parameter $\hat{\theta}$ that maximizes $L$; most or all work by stepping through the parameter space to obtain a sequence $\{\theta_i\}_{i \in \mathbb{N}}$ of parameters along which $L$ increases.  If the associated sequence of probability measures, $\{\textrm{Pr}_{\theta_i} \}_{i \in \mathbb{N}}$, converges towards a probably measure whose mass is unrealistically concentrated on a small number of graphs -- often the empty or complete graph -- the model is said to be \emph{degenerate} (\cite{handcock}). 

\begin{theorem} \label{t:degeneracy}
Suppose $t \not \in C$.   Then there exists a sequence of parameters $\{\theta_i\}_{i \in \mathbb{N}} \in \mathbb{R}^n$ such that 
\begin{enumerate}
\item $\lim_{i \to \infty} L(\theta_i) \to \infty$ and 
\item $\lim_{i \to \infty} \textrm{Pr}_{\theta_i} \left( \left \{g \in \mathcal{G}_k \mid z(g) \textrm{ is in the boundary of } C \right\} \right) = 1.$
\end{enumerate}
\end{theorem}

\begin{example} Here is a simple example that illustrates the content of Theorem \ref{t:degeneracy}.  Let $z(g)$ denote the single statistic that is the number of triangles in a graph $g \in \mathcal{G}_3$, the collection of graphs on $3$ vertices.  Of the $8$ graphs in $\mathcal{G}_3$, only one graph has one triangle, and the rest have $0$ triangles.  If $t$ is a nonrealizable statistic for the number of triangles, say $t = 5$, the expression for $L(\theta)$ becomes
$$L(\theta) = \frac{e^{\theta \cdot 5}}{7e^{\theta \cdot 0} + 1 e^{\theta \cdot 1}} = \frac{e^{5\theta}}{7+e^{\theta}}.$$  Hence $L(\theta) \to \infty$ as $\theta \to \infty$. 
\end{example} 
 
The qualitative meaning of $z(g)$ being in the boundary of $C$ is that the $z(g)$ is ``as extreme as possible'' in some way.  In \cite{handcock}, Handcock et al. describe the situation that  $t$ is in $C$ but is close to the boundary of $C$, calling it \emph{near degenerate} and stating that in this case Markov chain techniques can have poor convergence to the true maximum $\hat{\theta}$ (\cite{handcock}). Stability and near degeneracy are also investigated in \cite{Schweinberger}. 
Theorem \ref{t:degeneracy} suggests a related theoretical explanation for some cases of degeneracy observed during training -- namely, that if the set $C$ is ``skinny," then even a small error in estimating the coordinates of $t$ may push $t$ outside of $C$.  The most extreme case in this regard is that $\textrm{dim}(C) < n$, i.e. -- when the graph statistics that are the coordinate functions of $z$ are not linearly independent on $\mathcal{G}_k$.  However, even in cases where the graph statistics are linearly independent, there may be relationships among these statistics that dramatically constrain the width of $C$ in some direction.  A few theoretical results about how certain graph statistics are determined by other graph statistics, i.e. how the geometry of $C$ is constrained, are briefly mentioned at the end of Section \ref{s:implications}.

\section{Proof of Theorem~\ref{t:mainthm}} \label{s:concavityAndThm1}

The goal of this section is to justify Theorem \ref{t:mainthm}. Proposition \ref{p:concavity} establishes concavity of the log likelihood function and shows why strict concavity occurs if and only if $C$ is top-dimensional.   Lemma \ref{l:easydirection} shows (the easy direction) that if $L$ attains its supremum, the $t \in \textrm{rint}(C)$.  Lemma \ref{l:goodsector} is the technical engine that powers our proof of the opposite (harder) direction of Theorem \ref{t:mainthm}; its proof elucidates the geometrical importance of $t$ lying in the interior of $C$. Lemmas \ref{l:interiorImpliesMLE} and \ref{l:nontopdimcase} then prove this opposite direction in the cases that $C$ is and is not (respectively) top-dimensional, culminating in the proof of Theorem ~\ref{t:mainthm} at the end of the section.

Recall that a function $f:\mathbb{R}^n \to \mathbb{R}$ is said to be \emph{concave} if $f(tx_1+(1-t)x_2) \leq  tf(x_1) +(1-t)f(x_2)$ for all $x_1 \neq x_2$ in $\mathbb{R}^n$ and $t \in (0,1)$, and is \emph{strictly concave} if strict inequality holds.

\begin{proposition} \label{p:concavity}
The function $\ell$ is concave; furthermore $\ell$ is strictly concave if and only if $\textrm{dim}(C) = n$.  
\end{proposition}

Our proof of Proposition \ref{p:concavity} uses the following simple case of H\"{o}lder's Inequality (for a proof, see e.g. \cite{folland}):
Let $0<p,q<1$ satisfy $\tfrac 1 p +  \tfrac 1 q = 1$, let $X$ be a finite set, and let $f,g:X \to [0,\infty)$ be functions. Then 
\begin{equation}\label{eq:Holder}
\sum_{x \in X} f(x)g(x) \leq \left(\sum_{x \in X} f(x)^p\right)^{1/p} \left(\sum_{x \in X} g(x)^q \right)^{1/q},
\end{equation} with
 equality if and only if there exists a real number $c$ such that  $ f(x)^p = cg(x)^q$ for every $x \in X$.

\begin{proof}
Define $\kappa:\mathbb{R}^n \to \mathbb{R}$ by 
$$\kappa(\theta) \coloneqq \ln \left(\sum_{g \in \mathcal{G}_k} e^{\theta \cdot z(g)} \right).$$
Since $\ell(\theta) = \theta \cdot t - \kappa(\theta)$, it suffices to prove that $\kappa$ is always convex, and is strictly convex if and only if $\textrm{dim}(C) = n$.  

So fix any parameters $\theta_1,\theta_2 \in \mathbb{R}^n$ and a real number $0 < \tau < 1$.  Define the functions $F_{\theta_1,\tau},G_{\theta_2,\tau}:\mathcal{G}_k \to [0,\infty)$  by
$$F_{\theta_1,\tau}(g) \coloneqq e^{z(g) \cdot \theta_1 \tau}, \quad G_{\theta_2,\tau}(g) \coloneqq e^{z(g)\cdot \theta_2(1-\tau)}$$
Since $F_{\theta_1,\tau}$ and $G_{\theta_2,\tau}$ are nonnegative, H\"{o}lder's Inequality (\ref{eq:Holder}) using $p := \tfrac{1}{\tau}$ and $q := \tfrac{1}{1-\tau}$ gives
\begin{multline} \label{eq:HolderIneqApplication}
    \sum_{g \in \mathcal{G}_k} e^{z(g) \cdot (\theta_1\tau + \theta_2(1-\tau))}
    = \sum_{g \in \mathcal{G}_k} F_{\theta_1,\tau}(g)G_{\theta_2,\tau}(g)  \\ 
    \leq \left( \sum_{g \in \mathcal{G}_k} F_{\theta_1,\tau}(g)^{\frac{1}{\tau}} \right)^{\tau} \left( \sum_{g \in \mathcal{G}_k} G_{\theta_2,\tau}(g)^{\frac{1}{1-\tau}} \right)^{1-\tau} \\
=\left(\sum_{g \in \mathcal{G}_k }e^{z(g) \cdot \theta_1}\right)^\tau \left(\sum_{g \in \mathcal{G}_k }e^{z(g) \cdot \theta_2}\right)^{1-\tau},
\end{multline}
with equality if and only if the function $F_{\theta_1,\tau}^p/G_{\theta_2,\tau}^q$ is  constant  on $\mathcal{G}_k$.
Taking the logarithm of both sides yields
\begin{equation} \label{eq:kappaconvex}
\kappa(\theta_1 \tau + \theta_2 (1-\tau)) \leq \tau \kappa(\theta_1) + (1-\tau) \kappa(\theta_2),\end{equation} 
with equality  if and only if $F_{\theta_1,\tau}^p/G_{\theta_2,\tau}^q$ is a constant function.  

Thus, $\kappa$ is always convex, and is strictly convex if and only if for every $\theta_1 \neq \theta_2$ in $\mathbb{R}^n$ and $\tau \in (0,1)$, the function 
\begin{equation}\label{eq:quotient} 
g \mapsto \frac{F_{\theta_1,\tau}(g)^p}{G_{\theta_2,\tau}(g)^q} = e^{z(g) \cdot (\theta_1 - \theta_2)}
\end{equation} is nonconstant on $\mathcal{G}_k$.  
Note that for any $g_1,g_2 \in \mathcal{G}_k$ and $\theta_1 \neq \theta_2 \in \mathbb{R}^n$,  $z(g_1) \cdot (\theta_1 - \theta_2) = z(g_2) \cdot (\theta_1-\theta_2)$ if and only if $z(g_1)-z(g_2)$ is perpendicular to $\theta_1-\theta_2$.   Thus, there exist $\theta_1 \neq \theta_2$ in  $\mathbb{R}^n$ such that the function \eqref{eq:quotient} is constant on $\mathcal{G}_k$ if and only if $\textrm{Image}(z)$  is contained in an affine subset of $\mathbb{R}^n$ of dimension strictly less than $n$.  
\end{proof}

Since $L$ is concave by Proposition \ref{p:concavity}, appealing to the standard (e.g. \cite[Prop. 1.1.8]{bertsekas2009convex}) fact that a differentiable, concave function on $\mathbb{R}^n$ attains its (global) supremum at precisely the set of points where its gradient is $\vec{0}$  immediately yields the following corollary, which underlies our Theorem \ref{t:mainthm} proof strategy. 

\begin{corollary}  \label{c:gradientroots}
The function $L$ attains its supremum at $\hat{\theta}$ if and only if $\nabla L(\theta) = \vec{0}$.
\end{corollary}

\begin{lemma} \label{l:easydirection}
If $L$ attains its supremum, then $t \in \textrm{rint}(C)$. 
\end{lemma}

\begin{proof}
Suppose $L$ attains its supremum at $\hat{\theta}$; then $\ell = \ln L$ attains its supremum at $\hat{\theta}$.  Then by Corollary \ref{c:gradientroots}, 
$$\vec{0} = \nabla \ell(\hat{\theta})= t - \frac{1}{\sum_{g \in \mathcal{G}_k} e^{\hat{\theta} \cdot z(g)}} \sum_{g \in \mathcal{G}_k} e^{\hat{\theta} \cdot z(g)}z(g),$$ which implies 
\begin{equation} \label{eq:linearcombo}
t = \sum_{g \in \mathcal{G}_k}  \left( \frac{e^{\hat{\theta} \cdot z(g)}}{\sum_{g \in \mathcal{G}_k e^{\hat{\theta} \cdot z(g)}} } \right) z(g).
\end{equation}
Because the exponential functions is nonnegative, the coefficient on $z(g)$ in each term in \eqref{eq:linearcombo} is positive, and  the sum of these coefficients is $1$ by construction.  Therefore $t \in \textrm{rint}(C)$ (c.f. equation \eqref{eq:rintdef}).

\end{proof}

The following technical lemma will be key to proving the opposite (harder) direction of Theorem \ref{t:mainthm}.


\begin{lemma} \label{l:goodsector}
Suppose $\textrm{dim}(C) = n$ and $t \in \textrm{interior}(C)$.  Fix any unit vector $\gamma \in S := \{s \in \mathbb{R}^n : |s| = 1\}.$  Then there exists an open neighborhood $U$ in $S$ of $\gamma$, a graph $g_{\gamma} \in \mathcal{G}_k$ and a real number $\delta > 0$ such that 
$$ z(g_{\gamma}) \cdot u -  t \cdot u  \geq  \delta$$
for all $u \in U$. 
\end{lemma}

\begin{proof}
For any $s \in S$, denote by $H_s$ the hyperplane in $\mathbb{R}^n$ that has normal vector $s$ and passes through the point $t$.  Then $\mathbb{R}^n \setminus H_s$ consists of two open half-spaces; denote by $A_s^+$ the half-space that $s$ points into (i.e. that contains points of the form $t+\epsilon s$ for all $\epsilon > 0$).  

Because $t$ is in the interior $C$ and $C$ is top-dimensional, we may pick a vertex $v_{\gamma} \in A_{\gamma}^+$ of $C$. Also, we may pick a graph $g_\gamma \in \mathcal{G}_k$ such that $z(g_\gamma) = v_\gamma$.  This is because $C \setminus v_{\gamma}$ is a convex set that is strictly smaller than $C$, but $C$ is, by definition, the smallest convex set that contains $\{z(g) \mid g \in \mathcal{G}_k\}$, so $v_{\gamma}$ must be in that set.   Since $v_\gamma$ is in open half-space $A_{\gamma}^+$, there exists an open neighborhood $U$ of $\gamma$ in $S$ such that $v_{\gamma} \in A_u^+$ for all $u \in U$.  Furthermore, we may choose the neighborhood $U$ of $\gamma$ to be small enough that there exists some distance $\delta > 0$ such that $d(v_\gamma, H_u) \geq \delta$ for all $u \in U$. (See Figure \ref{f:sectorDiagram}.)

\begin{figure}
\centering
\includegraphics[width=.5\linewidth]{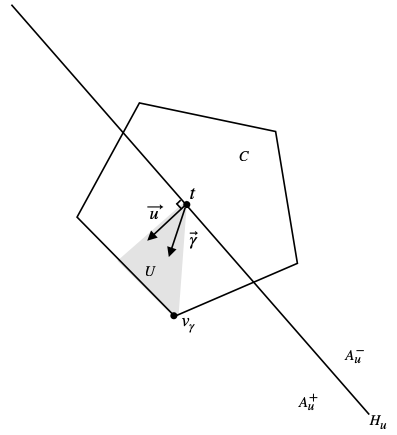}
\caption{Illustration of the proof of Lemma \ref{l:goodsector}. The shaded window $U$ of angles is constructed so that $v_{\gamma}$ is distance at least $\delta>0$ (in the positive $u$ direction) from $H_u$ for all $u \in U$. }
\label{f:sectorDiagram}
\end{figure}
 
 Recall that for any $x \in \mathbb{R}^n$ and unit vector $u$, $u \cdot x$ is the signed distance of $x$ in direction $u$ from the origin (technically, the scalar projection of $x$ onto $u$).  Since $t \in H_u$ and $u$ ``points into'' $A_u^+$, every point $x \in A_u^+$ satisfies $x \cdot u > t \cdot u$, and $x \cdot u - t \cdot u = d(x,H_u)$. Consequently, for all $u \in U$, $v_{\gamma} \cdot u - t \cdot u \geq \delta$.  
\end{proof}


\begin{lemma} \label{l:interiorImpliesMLE}
Suppose $C$ is top-dimensional and  $t \in \textrm{interior}(C)$.  Then there exists a unique $\hat{\theta} \in \mathbb{R}^n$ such that $\hat{\theta}$ maximizes $L$, i.e. $$L(\hat{\theta}) = \sup_{\theta \in \mathbb{R}^n} L(\theta).$$
\end{lemma}

\begin{proof}
Suppose $\sup_{\theta} L(\theta)$ is not attained; then  $\sup_{\theta} \ell(\theta)$ is not attained. Fix any constant $k < \sup_{\theta} \ell(\theta)$.  Since $\ell$ is continuous, it attains its maximum on any compact set.  Thus, there exists a sequence of points $\{\theta_j\}_{j \in \mathbb{N}}$ with $|\theta_j| \to \infty$ such that 
\begin{equation} \label{eq:boundedbelow}
\ell(\theta_j) \geq k
\end{equation} for all $j$.  
By passing to a subsequence, if necessary, we may assume that the sequence  $ \left \{ \frac{\theta_j }{|\theta_j |} \right \}_{j \in \mathbb{N}}$ converges (since $S\coloneqq\{x \in \mathbb{R}^n : |x| =1\}$ is compact). Set $\gamma \in S$ to be the vector
\begin{equation} \label{eq:gamma} 
\gamma := \lim_{j \to \infty} \frac{\theta_j }{|\theta_j |}.
\end{equation}

By Lemma \ref{l:goodsector}, there exists an open neighborhood $U$ of $\gamma$ in $S$,  a graph $g_{\gamma} \in \mathcal{G}_k$, and a constant $\delta > 0$ such that 
$$u \cdot (t - z(g_{\gamma})) < -d$$ for all $u \in U$. 

By \eqref{eq:gamma}, there exists $M \in \mathbb{N}$ such that  $\frac{\theta_j }{|\theta_j |} \in U$ whenever $j \geq M$.  Thus, whenever $j \geq M$, 
\begin{multline} \ell(\theta_j) = \theta_j \cdot t -  \ln \left( \sum_{g \in \mathcal{G}_k} e^{(\theta_j) \cdot z(g)} \right) 
\leq \theta_j \cdot t -  \ln \left(  e^{\theta_j \cdot z(g_{\gamma})} \right) \\ = \theta_j \cdot (t -z(\bar{g}))  =  |\theta_j| \frac{\theta_j}{|\theta_j|}  \cdot (t -z(g_{\gamma})) \leq |\theta_j| (-\delta).
 \end{multline}
But then 
$$\lim_{j \to \infty} \ell(\theta_j) \leq \lim_{j \to \infty} |\theta_j| (-\delta) = -\infty,$$
which contradicts \eqref{eq:boundedbelow}.  

For uniqueness, note that $\ell$ is strictly concave by Proposition \ref{p:concavity}, and it is straightforward to prove that a strictly concave function attains its maximum at at most one point. 

\end{proof}

\color{black}

We now consider the case that $C$ is not top-dimensional.  
Recall that $V$ denotes the smallest linear subspace of $\mathbb{R}^n$ that contains a translate of $C$, i.e. 
$$V \coloneqq  \textrm{span}(\{z(g_1) - z(g_2) \mid g_1,g_2 \in \mathcal{G}_k\})$$ and
set $V^{\perp} = \{u \in \mathbb{R}^n \mid u \cdot v = 0 \textrm{ for all } v \in V\}$.

\begin{lemma} \label{l:nontopdimcase}
Suppose $C$ is not top-dimensional and  $t \in \textrm{rint}(C)$.  Then there exists a unique $\hat{\theta} \in V \subset \mathbb{R}^n$ such that $\hat{\theta}$ maximizes $L$, i.e. $$L(\hat{\theta}) = \sup_{\theta \in \mathbb{R}^n} L(\theta).$$
\end{lemma}

\begin{proof}
Consider any $n_1,n_2 \in V^{\perp}$, $\epsilon_1,\epsilon_2 \in \mathbb{R}$, and $\theta \in \mathbb{R}^n$.
Algebraic manipulation and the fact that $z(g) - t \in V^{\perp}$ for all $g \in \mathcal{G}_k$ gives 
\begin{multline}(t+\epsilon_1n_1) \cdot (\theta + \epsilon_2n_2) - \ln  \left( \sum_{g \in \mathcal{G}_k} e^{(z(g)+\epsilon_1n_1) \cdot (\theta+\epsilon_2n_2)} \right) \\
 =t \cdot \theta -  \ln  \left( \sum_{g \in \mathcal{G}_k} e^{z(g) \cdot \theta+ (z(g)-t) \cdot \epsilon_2n_2) } \right) = t \cdot \theta -  \ln  \left( \sum_{g \in \mathcal{G}_k} e^{z(g) \cdot \theta} \right),
 \end{multline}
 i.e. $\ell$ is invariant under translating  $t$ and all points $z(g)$ by any fixed vector in $V^{\perp}$, as well as under changing the $V^{\perp}$ component of $\theta$.  Consequently, without loss of generality, we may assume that $C \subset V$ and consider only parameters $\theta$ that lie in $V$.  Thus, we may view $C$ as a top-dimensional polytope embedded in $V$, with $t$ in the interior of $C$, and apply Lemma \ref{l:interiorImpliesMLE}.
\end{proof}

\begin{proof}[Proof of Theorem~\ref{t:mainthm}] Lemma  \ref{l:easydirection} proves $t \in \textrm{rint}(C)$ is a necessary condition for the maximum likelihood to be attained. Lemmas \ref{l:interiorImpliesMLE} and \ref{l:nontopdimcase} prove the condition is sufficient and the uniqueness (in all of $\mathbb{R}^n$ if $C$ is top-dimensional, and in $V$ otherwise) of the maximizing parameter $\hat{\theta}$. 
\end{proof}



\section{Implications of approximate linear dependence of statistics for degeneracy} \label{s:implications}

In this section, we prove Theorem 2 and discuss its implications for the situation that  a modeler attempts to calibrate an ERGM model (i.e. find a value of $\hat{\theta}$ that maximizes $L$) based on a value of $t$ that is estimated from empirical data and lies outside of $C$.  For more on model degeneracy and, in particular, near degeneracy -- 
-- which refers to the situation that $t$ lies inside but near the boundary of $C$ -- we refer the interested reader to \cite{handcock}.  In the discussion after the proof of Theorem 2, we 
remark on a related possible explanation of some instances of degeneracy: if 
$C$ is ``skinny,'' even a tiny error estimating the coordinates of $t$ could push $t$ outside of $C$.  Moreover, results in extremal graph theory may provide theoretical insight into why some sets $C$; in particular, the number of triangles in a graph strongly constrains other properties of the graph, and thus using triangles as a graph statistic may lead to the set $C$ being ``skinny.''

 We prove the following (slightly more precise version) of Theorem 2:\\

\begin{thm2}
Suppose $t \not \in C$.  Let $H$ be any hyperplane in $\mathbb{R}^n$ that separates $t$ from $C$, and let $\theta$ be a normal vector to $H$ that points into the half space that contains $t$.  Then 
\begin{enumerate}[label=(\roman*)]
\item  \label{i:likelihoodInfty} $\lim_{r \to \infty} L(r\theta) = \infty$, and
\item  \label{i:extremegraphs} $\lim_{r \to \infty}  \textrm{Pr}_{r \theta} \{g \in \mathcal{G} \mid z(g) = \textrm{max}_{\bar{g} \in \mathcal{G}} \theta \cdot z(\bar{g})\} =1.$
\end{enumerate}
\end{thm2}

\medskip

The condition $z(g) = \textrm{max}_{\bar{g} \in \mathcal{G}} \theta \cdot z(g)$ in item \ref{i:extremegraphs} above implies that
$z(g)$ is in the relative (meaning viewing $C$ as a subset of its affine hull) boundary of $C$, i.e. the statistic $z(g)$ is ``as extreme as possible'' in some sense.  

\begin{proof}
The assumptions immediately imply $\theta \cdot t > \theta \cdot z(g)$ for all $g \in \mathcal{G}_k$ (since $t$ is ``farther'' in the positive $\theta$ direction than points in $C$ are).  Since $z(\mathcal{G}_k)$ is discrete, it follows that there exists $\epsilon > 0$ such that $z(g) \cdot \theta < (1-\epsilon) t \cdot \theta$ for all $g \in \mathcal{G}_k$.  

\emph{Case 1:} $\theta \cdot t > 0$. Then for any $r > 0$,  
\begin{multline*}
\ell(r\theta) = r\theta \cdot t - \ln \left( \sum_{g \in \mathcal{G}_k} e^{z(g) \cdot r\theta}\right) 
 \geq r\theta \cdot t - \ln \left( \sum_{g \in \mathcal{G}_k} e^{(1-\epsilon)r \theta \cdot t}\right)  \\
 = r\theta \cdot t - \ln |\mathcal{G}_k| - (1-\epsilon)r \theta \cdot t = r\epsilon \theta \cdot t - \ln |\mathcal{G}_k|.
\end{multline*}

\emph{Case 2:} $\theta \cdot t \leq 0$.  Then $z(g) \cdot \theta < 0$ for all $g \in \mathcal{G}_k$.  
Then for any $r > 0$ and any fixed $g_0 \in \mathcal{G}_k$, 
 $$\ell(r\theta) \geq - \ln \left( \sum_{g \in \mathcal{G}_k} e^{z(g) \cdot r\theta}\right) \geq - \ln (e^{z(g_0) \cdot r \theta})=-r z(g_0) \cdot \theta.$$
 In both cases, \ref{i:likelihoodInfty} follows immediately.

The limit \ref{i:extremegraphs} follows immediately from the fact that $z(\mathcal{G}_k)$ is a discrete set, and the weight $P_{r\theta}$ assigns to a graph $g$ is proportional to $e^{r\theta \cdot z(g)}$.
 \end{proof}

Theorem \ref{t:degeneracy} may be of particular potential importance to modelers when $C$ is ``skinny,'' since even a tiny error estimating the coordinates of $t$ could push $t$ outside of $C$.  When $\textrm{dim}(C) < n$ -- i.e. when the statistics that are the component functions of $z$ are not linearly independent -- the set $C$ has zero measure, so it is unlikely for an empirically estimated value of $t$ to lie in $C$.  An example of two graph statistics that can easily be seen to be linearly dependent are mean vertex degree and total number of edges (c.f. Figure \ref{f:rint}); while ERGMs practitioners know to avoid simultaneously using both of these statistics, there may be other collections of statistics that are linearly dependent but cannot be determined to be so. There is, to the best of our knowledge, no general method (other than directly computing a sufficient collection of graph statistics) for verifying that a collection of graph statistics is linearly independent.  

Even when a collection of graph statistics are linearly independent, it is possible that they are ``approximately linearly dependent,'' meaning that the value of one statistic is constrained to lie in a ``small'' interval by the values of the other statistics.  
Extremal graph theory provides a wealth of results governing  how big or small some graph statistic can be, given constraints on other statistics about the graph. For example, Tur\'{a}n's Theorem, a cornerstone of extremal graph theory, gives an upper bound on the number of edges that a graph with $n$ vertices that does not contain a $(r+1)$-clique can have.  For a survey of results in extremal graph theory, we refer the interested reader to \cite{bollobas2004extremal} and \cite{nikiforov2011some}.  
Restricting attention to empirically observed social network graphs, work by Faust asserts that the number of triangles is ``well-predicted by lower order graph features (density and dyads), accounting for around 90\% of the variability in triad distributions'' \cite{faust2010puzzle}.  Triangle counts are known to be associated with degenerate models  (\cite{Hunter2008ergm, goodreau2008statnet}).  
 On the algorithmic side, developing techniques for estimating the number of $k$-cliques in a graph based on partial information about the graph is an active area of inquiry (see. e.g. \cite{eden2020approximating}).  Approximate linear dependence of statistics, coupled with imperfectly estimated values of $t$, may be a mechanism through which degeneracy arises during model training. 
\color{black}

\bibliographystyle{plain}
\bibliography{MLEProof}

\def\polhk#1{\setbox0=\hbox{#1}{\ooalign{\hidewidth
  \lower1.5ex\hbox{`}\hidewidth\crcr\unhbox0}}}
\begin{thebibliography}{10}

\bibitem{Andalibi2021}
Ali Andalibi, Naoru Koizumi, Meng-Hao Li, and Abu~Bakkar Siddique.
\newblock Symptom and age homophilies in {SARS-CoV-2} transmission networks
  during the early phase of the pandemic in {J}apan.
\newblock {\em Biology}, 10:499, 6 2021.

\bibitem{barndorff2014information}
Ole Barndorff-Nielsen.
\newblock {\em Information and exponential families: in statistical theory}.
\newblock John Wiley \& Sons, 2014.

\bibitem{bertsekas2009convex}
Dimitri Bertsekas.
\newblock {\em Convex optimization theory}, volume~1.
\newblock Athena Scientific, 2009.

\bibitem{bollobas2004extremal}
B{\'e}la Bollob{\'a}s.
\newblock {\em Extremal graph theory}.
\newblock Courier Corporation, 2004.

\bibitem{brown1986fundamentals}
Lawrence~D. Brown.
\newblock {\em Fundamentals of statistical exponential families with
  applications in statistical decision theory}, volume~9 of {\em Institute of
  Mathematical Statistics Lecture Notes---Monograph Series}.
\newblock Institute of Mathematical Statistics, Hayward, CA, 1986.

\bibitem{Carnegie2012}
N~B Carnegie and M~Morris.
\newblock Size matters: concurrency and the epidemic potential of hiv in small
  networks.
\newblock {\em PLoS ONE}, 7:e43048, 2012.

\bibitem{Chen2020}
Jing Chen, Laura~M. Justice, Anna Rhoad-Drogalis, Tzu-Jung Lin, and Brook
  Sawyer.
\newblock Social networks of children with developmental language disorder in
  inclusive preschool programs.
\newblock {\em Child Development}, 91:471--487, 3 2020.

\bibitem{CM}
Imre Csisz\'{a}r and Franti\v{s}ek Mat\'{u}\v{s}.
\newblock Generalized maximum likelihood estimates for exponential families.
\newblock {\em Probab. Theory Related Fields}, 141(1-2):213--246, 2008.

\bibitem{eden2020approximating}
Talya Eden, Dana Ron, and C~Seshadhri.
\newblock On approximating the number of k-cliques in sublinear time.
\newblock {\em SIAM Journal on Computing}, 49(4):747--771, 2020.

\bibitem{10.1007/978-3-642-03367-4_25}
David Eppstein and Emma~S Spiro.
\newblock The h-index of a graph and its application to dynamic subgraph
  statistics.
\newblock In {\em Workshop on Algorithms and Data Structures}, pages 278--289.
  Springer, 2009.

\bibitem{faust2010puzzle}
Katherine Faust.
\newblock A puzzle concerning triads in social networks: Graph constraints and
  the triad census.
\newblock {\em Social Networks}, 32(3):221--233, 2010.

\bibitem{folland}
Gerald~B Folland.
\newblock {\em Real analysis: modern techniques and their applications},
  volume~40.
\newblock John Wiley \& Sons, 1999.

\bibitem{Geyer}
Charles~J. Geyer.
\newblock {Likelihood inference in exponential families and directions of
  recession}.
\newblock {\em Electronic Journal of Statistics}, 3:259 -- 289, 2009.

\bibitem{Goodreau2012_PUMA}
Steven~M Goodreau, Nicole~B Carnegie, Eric Vittinghoff, Javier~R Lama, Jorge
  Sanchez, Beatriz Grinsztejn, Beryl~A Koblin, Kenneth~H Mayer, and Susan~P
  Buchbinder.
\newblock What drives the {US} and {P}eruvian {HIV} epidemics in men who have
  sex with men ({MSM})?
\newblock {\em PLoS ONE}, 7:e50522, 2012.

\bibitem{goodreau2008statnet}
Steven~M Goodreau, Mark~S Handcock, David~R Hunter, Carter~T Butts, and Martina
  Morris.
\newblock A statnet tutorial.
\newblock {\em Journal of statistical software}, 24(9):1, 2008.

\bibitem{handcock}
Mark~S Handcock, Garry Robins, Tom Snijders, Jim Moody, and Julian Besag.
\newblock Assessing degeneracy in statistical models of social networks.
\newblock Technical report, Working paper, 2003.

\bibitem{statnet}
MS~Handcock, DR~Hunter, CT~Butts, SM~Goodreau, and M~Morris.
\newblock Ergm: Fit, simulate and diagnose exponential-family models for
  networks, version 2.1.
\newblock {\em URL: http://statnetproject. org}, 2003.

\bibitem{Hunter2008ergm}
David~R Hunter, Mark~S Handcock, Carter~T Butts, Steven~M Goodreau, and Martina
  Morris.
\newblock ergm: A package to fit, simulate and diagnose exponential-family
  models for networks.
\newblock {\em Journal of statistical software}, 24(3):nihpa54860, 2008.

\bibitem{Hancean2020}
Marian-Gabriel Hâncean, Mitja Slavinec, and Matjaž Perc.
\newblock The impact of human mobility networks on the global spread of
  {COVID-19}.
\newblock {\em Journal of Complex Networks}, 8, 3 2021.

\bibitem{Jenness2018}
Samuel~M. Jenness, Steven~M. Goodreau, and Martina Morris.
\newblock Epimodel : An {R} package for mathematical modeling of infectious
  disease over networks.
\newblock {\em Journal of Statistical Software}, 84, 2018.

\bibitem{Jenness2016}
Samuel~M Jenness, Steven~M Goodreau, Martina Morris, and Susan Cassels.
\newblock Effectiveness of combination packages for {HIV-1} prevention in
  sub-saharan africa depends on partnership network structure: a mathematical
  modelling study.
\newblock {\em Sexually Transmitted Infections}, 92:619--624, 12 2016.

\bibitem{jin2013fitting}
Ick~Hoon Jin and Faming Liang.
\newblock Fitting social network models using varying truncation stochastic
  approximation mcmc algorithm.
\newblock {\em Journal of computational and graphical statistics},
  22(4):927--952, 2013.

\bibitem{Khanna2014}
Aditya~S. Khanna, Dobromir~T. Dimitrov, and Steven~M. Goodreau.
\newblock What can mathematical models tell us about the relationship between
  circular migrations and hiv transmission dynamics?
\newblock {\em Mathematical Biosciences and Engineering}, 11:1065--1090, 6
  2014.

\bibitem{Khanna2019}
Aditya~S Khanna, John~A Schneider, Nicholson Collier, Jonathan Ozik, Rodal
  Issema, Angela di~Paola, Abigail Skwara, Arthi Ramachandran, Jeannette Webb,
  Russell Brewer, William Cunningham, Charles Hilliard, Santhoshini Ramani,
  Kayo Fujimoto, Nina Harawa, BARS~Study Group, Getting to~Zero IL
  Research~Evaluation, and Data~(RED) Committee.
\newblock A modeling framework to inform preexposure prophylaxis initiation and
  retention scale-up in the context of 'getting to zero' initiatives.
\newblock {\em AIDS (London, England)}, 33:1911--1922, 10 2019.

\bibitem{krivitsky2012package}
Pavel~N Krivitsky, Mark~S Handcock, David~R Hunter, and Maintainer Pavel~N
  Krivitsky.
\newblock Package ‘ergm. count’.
\newblock {\em Journal of Statistics}, 6:1100--1128, 2012.

\bibitem{nikiforov2011some}
Vladimir Nikiforov.
\newblock Some new results in extremal graph theory.
\newblock {\em arXiv preprint arXiv:1107.1121}, 2011.

\bibitem{rinaldo2009geometry}
Alessandro Rinaldo, Stephen~E Fienberg, and Yi~Zhou.
\newblock On the geometry of discrete exponential families with application to
  exponential random graph models.
\newblock {\em Electronic Journal of Statistics}, 3:446--484, 2009.

\bibitem{Robins2012}
Garry Robins, Jenny~M. Lewis, and Peng Wang.
\newblock Statistical network analysis for analyzing policy networks.
\newblock {\em Policy Studies Journal}, 40:375--401, 8 2012.

\bibitem{Schweinberger}
Michael Schweinberger.
\newblock Instability, sensitivity, and degeneracy of discrete exponential
  families.
\newblock {\em J. Amer. Statist. Assoc.}, 106(496):1361--1370, 2011.

\bibitem{Windzio2018}
Michael Windzio.
\newblock The network of global migration 1990–2013.
\newblock {\em Social Networks}, 53:20--29, 5 2018.

\bibitem{Wu2021}
Bing Wu and Cancan Wu.
\newblock Research on the mechanism of knowledge diffusion in the {MOOC}
  learning forum using {ERGMs}.
\newblock {\em Computers \& Education}, 173:104295, 11 2021.

\bibitem{Zelner2020}
Jon Zelner.
\newblock Is a more data-driven approach the future of tuberculosis
  transmission modeling?
\newblock {\em Clinical Infectious Diseases}, 70:2403--2404, 5 2020.

\end{thebibliography}

\end{document}